\documentclass[11pt,reqno]{amsart}
\usepackage{fullpage}
\usepackage{mathrsfs,amssymb,graphicx,verbatim,amsmath,amsfonts}
\usepackage{paralist}
\usepackage[breaklinks,pdfstartview=FitH]{hyperref}
\usepackage{upgreek}
\renewcommand{\le}{\leqslant}
\renewcommand{\ge}{\geqslant}

\renewcommand{\setminus}{\smallsetminus}
\renewcommand{\gamma}{\upgamma}

\newcommand{\dd}{\mathrm{d}}

\newcommand{\Rad}{\mathrm{\bf Rad}}

\newcommand{\n}{[n]}

\newcommand{\T}{\mathsf{T}}
\renewcommand{\d}{\delta}

\newcommand{\e}{\varepsilon}
\newcommand{\R}{\mathbb R}
\newcommand{\1}{\mathbf 1}

\newtheorem{theorem}{Theorem}
\newtheorem{lemma}[theorem]{Lemma}
\newtheorem{proposition}[theorem]{Proposition}

\newtheorem{corollary}[theorem]{Corollary}

\theoremstyle{remark}
\newtheorem{remark}[theorem]{Remark}

\renewcommand{\S}{\mathsf{S}}
\renewcommand{\subset}{\subseteq}
\newcommand{\E}{\mathbb{ E}}

\newcommand{\J}{\mathsf{J}}

\renewcommand{\hat}{\widehat}

\newcommand{\N}{\mathbb N}
\newcommand{\Z}{\mathbb Z}

\newcommand{\eqdef}{\stackrel{\mathrm{def}}{=}}

\newcommand{\A}{\mathsf{A}}
\newcommand{\EE}{\mathsf{E}}
\renewcommand{\emptyset}{\varnothing}

\begin{document}

\title[Sharp metric $X_p$ inequalities]{Discrete Riesz transforms and sharp metric $X_p$ inequalities}

\author{Assaf Naor}
\address{Mathematics Department\\ Princeton University\\ Fine Hall, Washington Road, Princeton, NJ 08544-1000, USA}
\email{naor@math.princeton.edu}
\thanks{Supported in part by the BSF, the Packard Foundation and the Simons Foundation.}

\date{}

\begin{abstract} For $p\in [2,\infty)$ the metric $X_p$ inequality with sharp scaling parameter is proven here to hold true in $L_p$.  The geometric consequences of this result include the following sharp statements about embeddings of $L_q$ into $L_p$ when $2<q<p<\infty$: the maximal $\theta\in (0,1]$ for which $L_q$ admits a bi-$\theta$-H\"older embedding into $L_p$ equals $q/p$, and for $m,n\in \N$ the smallest possible bi-Lipschitz distortion of any embedding into $L_p$ of the grid $\{1,\ldots,m\}^n\subset \ell_q^n$ is bounded above and below by constant multiples (depending only on $p,q$) of the quantity $\min\{n^{(p-q)(q-2)/(q^2(p-2))}, m^{(q-2)/q}\}$.
\end{abstract}

\maketitle

\section{Introduction}

The purpose of the present article is to resolve positively three conjectures that were posed by the author in collaboration with G. Schechtman in~\cite{NS14}. Specifically, we shall prove here that Conjecture~1.5, Conjecture~1.8 and Conjecture~1.12 of~\cite{NS14} all have a positive answer. As we shall explain below, of these three conjectures, Conjecture~1.8 was a  longstanding folklore open problem in embedding theory, while Conjecture~1.12 asserts the validity of a quite subtle and perhaps unexpected phase transition phenomenon that was first formulated as conceivably holding true in~\cite{NS14}.  Conjecture~1.5 relates to a bi-Lipschitz invariant that was introduced in~\cite{NS14}, asking about finer properties of this invariant in terms of a certain auxiliary parameter.

It was proven in~\cite{NS14} that Conjecture~1.8 and Conjecture~1.12 follow from Conjecture~1.5. Thus Conjecture~1.5  is the heart of the matter and the main focus of the present article, but we shall first describe all of the above conjectures since, by proving their validity, we establish delicate geometric phenomena related to the metric structure of $L_p$ spaces. In addition to these applications, a key contribution of the present article is the use of a deep result of Lust-Piquard~\cite{Lus98} for geometric purposes. While~\cite{NS14} proposed an approach to resolve the above conjectures, formulated as Question~6.1 in~\cite{NS14} and discussed at length in~\cite[Section~6]{NS14}, where it was shown to imply the above conjectures, we do not pursue this approach here, and indeed Question~6.1 of~\cite{NS14} remains open. Below we take a different route, yielding a novel connection between purely geometric questions and investigations in modern  harmonic analysis and operator algebras.

\subsection{Geometric statements} Following standard notation in Banach space theory and embedding theory (as in, say, \cite{LT77,Ost13}), for $n\in \N$ and $p\in [1,\infty)$ we let $\ell_p^n$ denote the space $\R^n$ equipped with the $\ell_p$ norm. When referring to the space $L_p$, we mean for concreteness the Lebesgue space $L_p(\R)$, though all of our new geometric results apply equally well to any infinite dimensional $L_p(\mu)$ space. The $L_p$ distortion of a metric space $(X,d_X)$, denoted $c_p(X)\in [0,\infty]$, is the infimum over those $D\in [0,\infty]$ for which there exists a mapping $f:X\to L_p$ that satisfies
\begin{equation*}\label{eq:distortion def}
\forall\, x,y\in X,\qquad d_X(x,y)\le \|f(x)-f(y)\|_{L_p}\le Dd_X(x,y).
\end{equation*}
$(X,d_X)$ is said to admit a bi-Lipschitz embedding into $L_p$ if $c_p(X)<\infty$.

Given $m,n\in \N$ and $q\in [1,\infty)$, the metric space whose underlying set is $\{1,\ldots,m\}^n$ (the $m$-grid in $\R^n$), equipped with the metric inherited from $\ell_q^n$, will be denoted below by $[m]_q^n$. It follows from the classical work~\cite{Pal36} of Paley, in combination with general principles related to differentiation of Lipschitz functions (see~\cite[Chapter~7]{BL00}), that if $2<q<p<\infty$ then $\lim_{n\to \infty} c_p(\ell_q^n)=\infty$. Since $[m]_q^n$ becomes ``closer" to $\ell_q^n$ as $m\to \infty$, one can apply an ultrapower argument (see~\cite{Hei80}) to deduce from this that $\lim_{m,n\to \infty} c_p([m]_q^n)=\infty$, but such reasoning does not yield information on the rate  of growth of $c_p([m]_q^n)$.  Effective estimates here follow from an alternative approach of Bourgain~\cite{Bou87} (with an improvement in~\cite{GNS12}), as well as the approach of~\cite{NS14}, but the resulting bounds are far from being sharp. Resolving Conjecture~1.12 of~\cite{NS14},  Theorem~\ref{thm:grid distortion} below computes the quantity $c_p([m]_q^n)$ up to constant factors that may depend on $p,q$ but not on $m,n$.

\begin{theorem}[Sharp evaluation of the $L_p$ distortion of $\ell_q^n$ grids]\label{thm:grid distortion} Suppose that $p,q\in [2,\infty)$ satisfy $q<p$. Then for every $m,n\in \N$ we have
\begin{equation}\label{eq:grid distortion formula}
c_p\big([m]_q^n\big)\asymp_{p,q} \min\left\{n^{\frac{(p-q)(q-2)}{q^2(p-2)}},m^{1-\frac{2}{q}}\right\}.
\end{equation}
\end{theorem}
In the statement of Theorem~\ref{thm:grid distortion}, as well as in what follows, we use standard asymptotic notation. Namely, the notation $a\lesssim b$ (respectively $a\gtrsim b$) stands for $a\le c b$ (respectively $a\ge c b$) for some universal constant $c\in (0,\infty)$. The notation $a\asymp b$ stands for $(a\lesssim b)\wedge (b\lesssim a)$. When we allow for implicit constants to depend on parameters, we indicate this by subscripts. Thus $a\lesssim_{p,q} b$ (respectively $a\gtrsim_{p,q} b$) means that there exists $c(p,q)\in (0,\infty)$ that may depend only on $p,q$ such that $a\le c(p,q)b $ (respectively $a\ge c(p,q) b$). The notation $a\asymp_{p,q} b$ stands for  $(a\lesssim_{p,q} b)\wedge (b\lesssim_{p,q} a)$.

Very few results at the level of precision of Theorem~\ref{thm:grid distortion} are known, and analogous questions are open even for some values of $p,q$ that are not covered by Theorem~\ref{thm:grid distortion}; see~\cite[Remark~1.13]{NS14} for more on this interesting topic. The asymptotic formula~\eqref{eq:grid distortion formula} expresses the statement that there exist two specific embeddings of $[m]_q^n$ into $L_p$ such that one of them is always the best possible embedding of $[m]_q^n$ into $L_p$, up to constant factors that do not depend on $m,n$. One of these embeddings arises from the work of Rosenthal~\cite{Ros70} (relying also on computations in~\cite{GPP80,FJS88}), and the other is due to Mendel and the author~\cite{MN06} (relying also on a construction from~\cite{Sch38}). These issues, including precise descriptions of the above two embeddings, are explained in detail in~\cite{NS14}.

The following immediate corollary of Theorem~\ref{thm:grid distortion} asserts that if $2<q<p<\infty$ and $m,n\in \N$ then the $L_p$ distortion of $[m]_q^n$ exhibits a phase transition at $m\asymp n^{(p-q)/(q(p-2))}$.

\begin{corollary}[Sharp phase transition of the $L_p$ distortion of $\ell_q^n$ grids] Suppose that $m,n\in \N$ and $p,q\in (2,\infty)$ satisfy $q<p$. Then
$$
m\gtrsim n^{\frac{p-q}{q(p-2)}}\implies c_p\big([m]_q^n\big)\asymp_{p,q} c_p\big(\ell_q^n\big),
$$
while as $n\to\infty$ we have
$$
m=o\left(n^{\frac{p-q}{q(p-2)}}\right)\implies c_p\big([m]_q^n\big)=o\left(c_p\big(\ell_q^n\big)\right).
$$
\end{corollary}

Thus, to state one concrete example so as to illustrate the situation whose validity we establish here, when, say, $q=3$ and $p=4$, and one tries to embed the grid $[m]_3^n$ into $L_4$, one sees that there is a phase transition at $m\asymp \!\sqrt[6]{n}$. If $m\gtrsim \!\sqrt[6]{n}$ then any embedding of $[m]_3^n$ into $L_4$ incurs the same  distortion (up to universal constant factors) as the distortion required to embed all of $\ell_3^n$ into $L_4$, which grows like $\!\!\sqrt[18]{n}$. However, if $m=o(\!\sqrt[6]{n})$ then one can embed $[m]_3^n$ into $L_4$ with distortion $o(\!\!\sqrt[18]{n})$, and in this case the $L_4$ distortion of $[m]_3^n$ is $\sqrt[3]{m}$, up to universal constant factors.

\smallskip
Our second geometric result is Theorem~\ref{thm:snowflake sharp} below, which resolves Conjecture~1.8 of~\cite{NS14}.

\begin{theorem}[Evaluation of the critical $L_p$ snowflake exponent of $L_q$]\label{thm:snowflake sharp} Suppose that $p,q\in (2,\infty)$ satisfy $q<p$. Then the maximal $\theta\in (0,1]$ for which the metric space $(L_q,\|x-y\|_{L_q}^\theta)$ admits a bi-Lipschitz embedding into $L_p$ equals $q/p$.
\end{theorem}
In the setting of Theorem~\ref{thm:snowflake sharp}, the fact that the metric space  $(L_q,\|x-y\|_{L_q}^{q/p})$ does indeed admit a bi-Lipschitz (even isometric) embedding into $L_p$ was established by Mendel and the author  in~\cite{MN04}. Since then, it has been a well known conjecture that in this context the H\"older exponent $q/p$ cannot be increased, but before~\cite{NS14} it wasn't even known that if $(L_q,\|x-y\|_{L_q}^\theta)$ admits a bi-Lipschitz embedding into $L_p$ then necessarily $\theta<1-\delta$ for some $\d=\d(p,q)>0$. Note that the endpoint case $q=2$ must be removed from Theorem~\ref{thm:snowflake sharp} since $L_2$ embeds isometrically into $L_p$.

\subsection{Optimal scaling in the $L_p$-valued metric $X_p$ inequality} In what follows, given $n\in \N$ we shall denote the set $\{1,\ldots,n\}$ by $\n$. The coordinate basis of $\R^n$ will be denoted by $e_1,\ldots,e_n$, and for a sign vector $\e=(\e_1,\ldots,\e_n)\in \{-1,1\}^n$ and a subset $\S\subset \n$ we shall use the notation
\begin{equation}\label{eq:def eps S}
\e_{\S}\eqdef \sum_{j\in \S}\e_je_j.
\end{equation}

Fix $p\in (0,\infty)$. Following~\cite{NS14}, a metric space $(X,d_X)$ is said to be an $X_p$ metric space if there exists $\mathfrak{X}\in (0,\infty)$ such that for every $n\in \N$ and $k\in \n$ there exists $m\in \N$ such that every function $f:\Z_{2m}^n\to X$ satisfies the following distance inequality.
\begin{multline}\label{eq:metric space is X_p beginning def2}
\bigg(\frac{1}{\binom{n}{k}}\sum_{\substack{S\subset
\n\\|S|= k}}
\E \Big[d_X\big(f(x+m\e_S),f(x)\big)^p\Big]\bigg)^{\frac{1}{p}}\\\le \mathfrak{X}m\Bigg(
\frac{k}{n}\sum_{j=1}^n\E\Big[d_X\big(
f(x+e_j),f(x)\big)^p\Big]+\left(\frac{k}{n}\right)^{\frac{p}{2}}
\E\Big[d_X\big(f(x+ \e),f(x)\big)^p\Big]\bigg)^{\frac{1}{p}}.
\end{multline}
The expectations in~\eqref{eq:metric space is X_p beginning def2} are with respect to $(x,\e)\in \Z_{2m}^n\times \{-1,1\}^n$ chosen uniformly at random. We refer to~\cite{NS14} for a detailed discussion of the meaning of~\eqref{eq:metric space is X_p beginning def2}; see also Sections~\ref{sec:how to apply}, \ref{sec:intro chaos}  below.

The above definition of $X_p$ metric  spaces introduces the auxiliary integer $m\in \N$, which we call the {\em scaling parameter} corresponding to $n$ and $k$. For some purposes $m$ can be allowed to be arbitrary, but for other purposes one needs to obtain good bounds on $m$ (as a function of $n,k$). It can, however, be quite difficult to obtain sharp bounds on scaling parameters in metric inequalities (for example, an analogous question in the context of metric cotype~\cite{MN08} is longstanding and important). In~\cite{NS14} it was proven that if $p\in [2,\infty)$ then $L_p$ is an $X_p$ metric space. The proof in~\cite{NS14} yields the validity of~\eqref{eq:metric space is X_p beginning def2} when $X=L_p$ whenever $m\gtrsim_p n^{3/2}/\sqrt{k}$. It was also shown in~\cite[Proposition~1.4]{NS14} that if $p\in (2,\infty)$ and $k$ is sufficiently large (as a function of $p$) then for~\eqref{eq:metric space is X_p beginning def2} to hold true in $L_p$ one must necessarily have $m\gtrsim_p \sqrt{n/k}$. Conjecture~1.5 of~\cite{NS14} asks whether for every $p\in (2,\infty)$ this lower bound on $m$ actually expresses the asymptotic behavior of the best possible scaling parameter, i.e., whether the metric $X_p$ inequality~\eqref{eq:metric space is X_p beginning def2} holds true in $L_p$ for every $m\gtrsim_p \sqrt{n/k}$. Theorem~\ref{thm:main} below resolves this conjecture positively.

\begin{theorem}[$L_p$ is an $X_p$ metric space with sharp scaling parameter]\label{thm:main} Suppose that $k,m,n\in \N$ satisfy $k\in \n$ and $m\ge \sqrt{n/k}$. Suppose also that $p\in [2,\infty)$. Then every $f:\Z_{8m}^n\to L_p$ satisfies
\begin{multline}\label{eq:desired main}
\bigg(\frac{1}{\binom{n}{k}}\sum_{\substack{\S\subset \n\\ |\S|=k}}\E\Big[\left\|f(x+4m\e_{\S})-f(x)\right\|_{L_p}^p\Big]\bigg)^{\frac{1}{p}}
\\\lesssim_p m\bigg(\frac{k}{n}\sum_{j=1}^n\E\Big[\|f(x+e_j)-f(x)\|_{L_p}^p\Big]+\left(\frac{k}{n}\right)^{\frac{p}{2}}\E\Big[\|f(x+\e)-f(x)\|_{L_p}^p\Big]\bigg)^{\frac{1}{p}},
\end{multline}
where the expectations are taken with respect to $(x,\e)\in \Z_{8m}^n\times \{-1,1\}^n$ chosen uniformly at random.
\end{theorem}

\begin{remark}
Our proof of Theorem~\ref{thm:main} shows that the implicit constant in~\eqref{eq:desired main} is $O(p^4/\log p)$. As explained in~\cite{NS14}, this constant must be at least a (universal) constant multiple of $p/\log p$. While it is conceivable that a more careful implementation of our approach could somewhat decrease the dependence on $p$ that we obtain, it seems that a new idea is required in order to establish the sharp dependence of $O(p/\log p)$ in~\eqref{eq:desired main} (if true). We leave the question of determining the correct asymptotic dependence on $p$ in~\eqref{eq:desired main} as an interesting (and perhaps quite challenging) open question.
\end{remark}

\subsubsection{Applications of Theorem~\ref{thm:main}}\label{sec:how to apply} The usefulness of the metric $X_p$ inequality for $L_p$ stems in part from the fact that it allows one to rule out the existence of metric embeddings in situations where the classical differentiation techniques fail. Examples of such situations include the treatment of discrete sets as in Theorem~\ref{thm:grid distortion}, where it isn't clear how to interpret the notion of derivative, as well as the treatment of H\"older mappings as in Theorem~\ref{thm:snowflake sharp}, where, unlike the Lipschitz case, mappings need not have any point of differentiability. In fact, by~\cite[Theorem~1.14]{NS14} both Theorem~\ref{thm:grid distortion} and Theorem~\ref{thm:snowflake sharp} follow from Theorem~\ref{thm:main}. For completeness, we shall now briefly sketch why this is so.

Suppose that $2\le q<p<\infty$ and $m,n\in \N$. It is simple to check, as done in~\cite[Lemma~3.1]{NS14}, that there exists $h:\Z_{8m}^n\to [32m]^n_q$ such that for $(x,\e)\in \Z_{8m}^n\times \{-1,1\}^n$, $\S\subset \n$ and $j\in \n$,
\begin{equation*}\label{eq:h properties}
\|h(x+4m\e_\S)-h(x)\|_{\ell_q^n}\asymp m|\S|^{\frac{1}{q}} \quad\mathrm{and}\quad \|h(x+e_j)-h(x)\|_{\ell_q^n}\asymp 1\quad\mathrm{and}\quad \|h(x+\e)-h(x)\|_{\ell_q^n}\asymp n^{\frac{1}{q}}.
\end{equation*}
Fix $D\in [1,\infty)$ and suppose that $\phi:[32m]^n_q\to L_p$ satisfies $\|x-y\|_{\ell_q^n}\le \|\phi(x)-\phi(y)\|_{L_p}\le D\|x-y\|_{\ell_q^n}$ for every $x,y\in [32m]_q^n$. An application of Theorem~\ref{thm:main} to $f=h\circ \phi$ (with $m$ replaced by $4m$), which we are allowed to do only when $k\in \n$ is  such that $4m\ge \sqrt{n/k}$, yields the bound
\begin{equation}\label{eq:max k 1}
D\gtrsim_p \max_{\substack{k\in \n\\ k\ge n/(16m^2)}}\frac{k^{\frac{1}{q}}}{\left(k+k^{\frac{p}{2}}n^{\frac{p}{q}-\frac{p}{2}}\right)^{\frac{1}{p}}}.
\end{equation}
By evaluating the maximum in~\eqref{eq:max k 1}, one arrives at the asymptotic lower bound on $c_p([32m]_q^n)$ that appears in~\eqref{eq:grid distortion formula}. As we explained earlier, the matching upper bound in~\eqref{eq:grid distortion formula} corresponds to the better of two explicit embeddings that are described in equations (11) and (27) of~\cite{NS14}.  This completes the deduction of Theorem~\ref{thm:grid distortion}. Next, fix $L\in [1,\infty)$ and $\theta\in (0,1]$. Suppose that $\psi:L_q\to L_p$ satisfies $\|x-y\|_{L_q}^\theta\le \|\psi(x)-\psi(y)\|_{L_p}\le L\|x-y\|_{L_q}^\theta$ for every $x,y\in L_q$. For $k,n\in \N$ with $k\in \n$, fix $m=\lceil \sqrt{n/(2k)}\rceil$ and apply Theorem~\ref{thm:main} to $f=\psi\circ h$. The estimate thus obtained is
$$
\left(\frac{n}{k}\right)^{\frac{\theta}{2}} k^{\frac{\theta}{q}}\lesssim_p L\sqrt{\frac{n}{k}}\left(k+\left(\frac{k}{n}\right)^{\frac{p}{2}} n^{\frac{\theta p}{q}}\right)^{\frac{1}{p}}.
$$
Hence, for every $n\in \n$ we have
\begin{equation}\label{eq:max k 2}
1\lesssim_p Ln^{\frac{1-\theta}{2}}\cdot  \min_{k\in \n} \left(k+k^{\frac{p}{2}} n^{p\left(\frac{\theta}{q}-\frac12\right)}\right)^{\frac{1}{p}}k^{\theta\left(\frac12-\frac{1}{q}\right)-\frac12}.
\end{equation}
Theorem~\ref{thm:snowflake sharp} now follows by choosing the optimal $k$ in~\eqref{eq:max k 2} and letting $n\to\infty$; complete details of this computation appear in the proof of Theorem~1.14 in~\cite{NS14}.

\subsection{Hypercube Riesz transforms and an $X_p$ inequality for Rademacher chaos}\label{sec:intro chaos} Fixing $n\in \N$, for every $h:\{-1,1\}^n\to \R$ and $j\in \n$ let $\partial_j h:\{-1,1\}^n\to \R$ be given by
\begin{equation}\label{eq:def partial j}
\forall\, \e\in \{-1,1\}^n,\qquad \partial_j h(\e)\eqdef h(\e)-h(\e_1,\ldots,\e_{j-1},-\e_j,\e_{j+1},\ldots,\e_n).
\end{equation}
Also, given $\S\subset \n$ we shall denote by $\EE_\S f:\{-1,1\}^n\to \R$ the function that is obtained from $h$ by averaging over the coordinates in $\S$, i.e.,
recalling the notation~\eqref{eq:def eps S}, we define
\begin{equation}\label{eq:def E operator}
\forall\, \e\in \{-1,1\}^n,\qquad \EE_{\S}h(\e)\eqdef \frac{1}{2^n} \sum_{\d\in \{-1,1\}^n} h\left(\d_{\S}+\e_{\n\setminus \S}\right).
\end{equation}
In particular, $\EE_\S h$ depends only on those entries of $\e\in \{-1,1\}^n$ that belong to $\n\setminus \S$. Given $p\in [1,\infty)$, we shall reserve from now on the notation $\|h\|_p$ exclusively for the $L_p$ norm of $h$ with respect to the {\em normalized} counting measure on the discrete hypercube $\{-1,1\}^n$, i.e.,
$$
\|h\|_p\eqdef \bigg(\frac{1}{2^n}\sum_{\e\in \{-1,1\}^n} |h(\e)|^p\bigg)^{\frac{1}{p}}=\left(\EE_{\n} |h|^p\right)^{\frac{1}{p}}.
$$
In what follows, $L_p^0(\{-1,1\}^n)$ denotes the subspace of all those $h\in L_p(\{-1,1\}^n)$ with $\EE_{\n}h=0$.

We shall work with the usual Fourier--Walsh expansion of a function $h:\{-1,1\}^n\to \R$. Thus, for every $\A\subset\n $ consider the corresponding Walsh function $W_\A:\{-1,1\}^n\to \R$ given by
$$
\forall\, \e\in \{-1,1\}^n,\qquad W_\A(\e)\eqdef \prod_{j\in \A} \e_j,
$$
and denote
$$
\hat{h}(\A)\eqdef\frac{1}{2^n}\sum_{\e\in \{-1,1\}^n} h(\e)W_\A.
$$
Then we have
$$
\forall\,\e\in \{-1,1\}^n,\qquad h(\e)=\frac{1}{2^n}\sum_{\A\subset \n} \hat{h}(\A)W_\A(\e).
$$

In probabilistic terminology, the above representation of $h$ as a multilinear polynomial in the variables $\e_1,\ldots,\e_n$ expresses it as Rademacher chaos. A useful inequality for Rademacher chaos of the first degree, i.e., for weighted sums of i.i.d. Bernoulli random variables, served as the inspiration for the metric $X_p$ inequality~\eqref{eq:metric space is X_p beginning def2}. Specifically, \eqref{eq:metric space is X_p beginning def2} is a nonlinear extension of the following inequality, which holds true for every $p\in [2,\infty)$, $k,n\in \N$ with $k\in \n$, and every $a_1,\ldots,a_n\in \R$.
\begin{equation}\label{eq:linearXp}
\bigg(\frac{1}{2^n\binom{n}{k}}\sum_{\substack{\S\subset \n\\ |\S|=k}}\sum_{\e\in \{-1,1\}^n}\Big|\sum_{j\in \S}\e_ja_j\Big|^p\bigg)^{\frac{1}{p}}\lesssim \frac{p}{\log p}\bigg(\frac{k}{n}\sum_{j=1}^n |a_j|^p+\frac{(k/n)^{\frac{p}{2}}}{2^n}\sum_{\e\in \{-1,1\}^n}\Big|\sum_{j=1}^n \e_ja_j\Big|^p\bigg)^{\frac{1}{p}}.
\end{equation}
This inequality is due to Johnson, Maurey, Schechtman and Tzafriri, who proved it in~\cite{JMST79} with a constant factor that grows to $\infty$ with $p$  faster than the $p/\log p$ factor that appears in~\eqref{eq:linearXp}. The factor $p/\log p$ that is stated in~\eqref{eq:linearXp} is best possible; in the above sharp form, \eqref{eq:linearXp} is due to Johnson, Schechtman and Zinn~\cite{JSZ85}. As a step towards Theorem~\ref{thm:main}, we shall  prove the following theorem in  Section~\ref{sec:chaos} below, thus extending~\eqref{eq:linearXp} to Rademacher chaos of arbitrary degree.

 \begin{theorem}[$X_p$ inequality for Rademacher chaos]\label{thm:cube h}
Suppose that $p\in [2,\infty)$, $n\in \N$ and $k\in \n$. Then every $h\in L_p^0(\{-1,1\}^n)$ satisfies
\begin{equation}\label{eq:on cube p implicit}
\bigg(\frac{1}{\binom{n}{k}} \sum_{\substack{\S\subset \n\\ |\S|=k}}\left\|\mathsf{E}_{\n\setminus \S} h \right\|_{p}^p\bigg)^{\frac{1}{p}} \lesssim_p
\bigg(\frac{k}{n} \sum_{j=1}^n \|\partial_j h\|_{p}^p +\left(\frac{k}{n}\right)^{\frac{p}{2}} \|h\|_{p}^p\bigg)^{\frac{1}{p}}.
\end{equation}
\end{theorem}
The deduction of Theorem~\ref{thm:main} from Theorem~\ref{thm:cube h} appears in Section~\ref{sec:implication} below.

\begin{remark}
As in~\eqref{eq:desired main}, the implicit constant that we obtain in~\eqref{eq:on cube p implicit} is $O(p^4/\log p)$. In fact, our proof yields the following slightly more refined estimate in the setting of Theorem~\ref{thm:cube h}.
\begin{equation}\label{eq:on cube bext we can do in terms of p}
\bigg(\frac{1}{\binom{n}{k}} \sum_{\substack{\S\subset \n\\ |\S|=k}}\left\|\mathsf{E}_{\n\setminus \S} h \right\|_{p}^p\bigg)^{\frac{1}{p}} \lesssim \frac{p^{\frac{5}{2}}}{\sqrt{\log p}}\left(\frac{k}{n}\right)^{\frac{1}{p}}\bigg(\sum_{j=1}^n \|\partial_j h\|_{p}^p\bigg)^{\frac{1}{p}}+\frac{p^4}{\log p}\sqrt{\frac{k}{n}}\cdot \|h\|_{p} .
\end{equation}
It remains open to determine the growth rate as $p\to \infty$ of the implicit constant in~\eqref{eq:on cube p implicit}.
\end{remark}

\subsubsection{Lust-Piquard's work} Our proof of Theorem~\ref{thm:cube h} uses deep work~\cite{Lus98} of Lust-Piquard on dimension-free bounds for discrete Riesz transforms. Specifically, for every $h:\{-1,1\}^n\to \R$ and $j\in \n$ the $j$th (hypercube) Riesz transform of $h$, denoted $\mathsf{R}_jh:\{-1,1\}^n\to \R$, is defined as follows.
\begin{equation}\label{eq:def Rj}
\forall\, \e\in \{-1,1\}^n,\qquad \mathsf{R}_jh(\e)\eqdef \sum_{\substack{\A\subset \n\\ j\in \A}} \frac{\hat{h}(\A)}{\sqrt{|\A|}}W_{\A}(\e).
\end{equation}
Lust-Piquard proved the following inequalities, which hold true for $p\in [2,\infty)$ and $h\in L_p^0(\{-1,1\}^n)$.
\begin{equation}\label{eq:LP square function}
\frac{1}{p^{3/2}} \|h\|_{p}\lesssim \bigg\|\Big(\sum_{j=1}^n (\mathsf{R}_jh)^2\Big)^{\frac12}\bigg\|_{p}\lesssim p \|h\|_{p}.
\end{equation}

The inequalities in~\eqref{eq:LP square function} were proved by Lust-Piquard in~\cite{Lus98}, though with a dependence on $p$ that is worse than what we stated above. The dependence on $p$ that appears in~\eqref{eq:LP square function} follows from~\cite{EL08}. Note that these estimates are stated in~\cite{EL08} in terms of the strong $(p,p)$  norm of the Hilbert transform with values in the Schatten--von Neumann trace class $\S_p$, but this norm was shown to be $O(p)$ by Bourgain in~\cite{Bou-hilbert}, and the bounds that we stated in~\eqref{eq:LP square function} result from a direct substitution of Bourgain's bound into the statements in~\cite{EL08}.

The availability of dimension independent bounds for Riesz transforms is a well known paradigm in other (non-discrete) settings, originating from important classical work of Stein~\cite{Ste83} in the case of $\R^n$ equipped with Lebesgue measure (see also~\cite{GV79,DR85,Ban86}). Most pertinent to the present context is the classical theorem of P.~A. Meyer~\cite{Mey84} (see also~\cite{Gun86}) that obtained dimension independent bounds for the Riesz transforms that are associated to $\R^n$ equipped with the Gaussian measure (and the Ornstein--Uhlenbeck operator). Pisier discovered in~\cite{Pis88} an influential alternative proof of P.~A. Meyer's theorem, based on a transference argument (see~\cite{CW76}) that allows one to reduce the question to the boundedness of the (one dimensional) Hilbert transform.

Lust-Piquard's work generally follows Pisier's strategy, but it also uncovers a phenomenon that is genuinely present in the hypercube setting and {\em not} in the Gaussian setting. Specifically, Lust-Piquard reduces the task of bounding the hypercube Riesz transforms to that of bounding the $\S_p$ norm of certain operators in a noncommutative $*$ algebra of ($2^n$ by $2^n$) matrices, and proceeds to do so using operator-theoretic methods, including her noncommutative Khinchine inequalities~\cite{Lus86}.

This indicates why the $\S_p$-valued Hilbert transform makes its appearance in Lust-Piquard's inequality (recall the paragraph above, immediately following~\eqref{eq:LP square function}), despite the fact that~\eqref{eq:LP square function} deals with real-valued functions on the (commutative) hypercube. Significantly, while the classical results on Riesz transforms (with respect to either Lebesgue measure or the Gaussian measure) yield dimension independent bounds for every $p\in (1,\infty)$, it turns out that~\eqref{eq:LP square function} actually fails to hold true when $p\in (1,2)$, as explained in~\cite{Lus98} (where this observation is attributed to unpublished work of Lamberton); see also~\cite[Section~5.5]{EL08}. The reason for this disparity between the ranges $p\in (1,2)$ and $p\in [2,\infty)$ becomes clear when one transfers the question to the noncommutative setting, and this suggests a more complicated (but still dimension-free) replacement for~\eqref{eq:LP square function} in the range $p\in (1,2)$, which Lust-Piquard also proved in~\cite{Lus98}. So, while it is conceivable that a proof of~\eqref{eq:LP square function} could be found that does not proceed along Lust-Piquard's noncommutative route, such a proof has not been found to date, and the qualitative divergence between the discrete situation and its continuous counterparts indicates that there may be an inherently different phenomenon at play here. Since its initial publication, Lust-Piquard's work influenced developments by herself and others that focused on proving related inequalities in other situations; we do not have anything new to add to this interesting body of work other than showing here that in addition to their intrinsic interest, such results can have a decisive role in understanding geometric embedding questions.

\section{Deduction of Theorem~\ref{thm:main} from Theorem~\ref{thm:cube h}}\label{sec:implication}

Assuming the validity of Theorem~\ref{thm:cube h} for the moment, we shall now proceed to show how it implies Theorem~\ref{thm:main}. Note that since~\eqref{eq:desired main} involves only the $p$th powers of distances in $L_p$, by integration it suffices to prove Theorem~\ref{thm:main} for real valued functions. So, from now on we shall assume that $m,n\in \N$ and we are given a function $f:\Z_{8m}^n\to \R$, the goal being to prove the validity of~\eqref{eq:desired main} for every $k\in \n$ provided that $m\ge \sqrt{n/k}$, with the $L_p$ norms replaced by absolute values in $\R$.

In what follows, given $\S\subset \n$ and $f:\Z_{8m}^n\to \R$, define a function $\T_{\! \S}f: \Z_{8m}^n\to \R$ by
\begin{equation}\label{eq:TS def}
\forall\, x\in \Z_{8m}^n,\qquad \T_{\! \S}f(x)\eqdef \frac{1}{2^n}\sum_{\d\in \{-1,1\}^n} f(x+2\d_\S).
\end{equation}
We record for future use the following simple lemma.

\begin{lemma}\label{lem:T close to identity}
For every $p\in [1,\infty)$, $m,n\in \N$, $\S\subset \n$ and $f:\Z_{8m}^n\to \R$ we have
\begin{equation}\label{eq:T factor 2}
\bigg(\frac{1}{(8m)^n} \sum_{x\in \Z_{8m}^n} |f(x)-\T_{\! \S}f(x)|^p\bigg)^{\frac{1}{p}}\le 2\bigg(\frac{1}{(16m)^n}\sum_{\e\in \{-1,1\}^n} \sum_{x\in \Z_{8m}^n} |f(x+\e)-f(x)|^p\bigg)^{\frac{1}{p}}.
\end{equation}
\end{lemma}
\begin{proof}
By convexity, for every $x\in \Z_{8m}^n$ we have
\begin{multline}\label{eq:simple triangle}
|f(x)-\T_{\! \S}f(x)|^p\le \frac{1}{2^n} \sum_{\d\in \{-1,1\}^n} \left|f(x)-f(x+2\d_{\S})\right|^p  \\\le \frac{2^{p-1}}{2^n} \sum_{\d\in \{-1,1\}^n} \Big(\left|f(x)-f(x+\d_{\S}+\d_{\n\setminus \S})\right|^p+ \left|f(x+\d_{\S}+\d_{\n\setminus \S})-f(x+2\d_{\S})\right|^p\Big).
\end{multline}
The desired estimate~\eqref{eq:T factor 2} follows by averaging~\eqref{eq:simple triangle} over $x\in \Z_{8m}^n$ while using the translation invariance of the uniform measure on $\Z_{8m}^n$, and that if $\d$ is uniformly distributed over $\{-1,1\}^n$ then the sign vectors  $\d_\S+\d_{\n\setminus \S}$ and $-\d_\S+\d_{\n\setminus \S}$ are both also uniformly distributed over $\{-1,1\}^n$.
\end{proof}

\begin{lemma}\label{lem:T version with scaling} Suppose that $m,n\in \N$ and $k\in \n$. If $p\in [2,\infty)$ then every $f:\Z_{8m}^n\to \R$ satisfies
\begin{multline}\label{eq:T version}
\frac{1}{(16m)^n\binom{n}{k}}\sum_{\substack{\S\subset \n\\ |\S|=k}}\sum_{\e\in \{-1,1\}^n}\sum_{x\in \Z_{8m}^n}\frac{\left|\T_{\!\n\setminus \S}f(x+4m\e_{\S})-\T_{\!\n\setminus \S}f(x)\right|^p}{m^p}\\
\lesssim_p \frac{k/n}{(8m)^n}\sum_{j=1}^n\sum_{x\in \Z_{8m}^n}|f(x+e_j)-f(x)|^p+\frac{(k/n)^{\frac{p}{2}}}{(16m)^n}\sum_{\e\in \{-1,1\}^n}\sum_{x\in \Z_{8m}^n} |f(x+\e)-f(x)|^p.
\end{multline}
\end{lemma}

\begin{proof} For every fixed $\S\subset \n$ we have
\begin{align}
\nonumber \bigg(\frac{1}{(16m)^n}&\sum_{\e\in \{-1,1\}^n}\sum_{x\in \Z_{8m}^n}\left|\T_{\!\n\setminus \S}f(x+4m\e_{\S})-\T_{\!\n\setminus \S}f(x)\right|^p\bigg)^{\frac{1}{p}}\\ \label{eq:before x to y}
&\le \sum_{k=1}^m \bigg(\frac{1}{(16m)^n}\sum_{\e\in \{-1,1\}^n}\sum_{x\in \Z_{8m}^n}\left|\T_{\!\n\setminus \S}f(x+4k\e_{\S})-\T_{\!\n\setminus \S}f(x+4(k-1)\e_{\S})\right|^p\bigg)^{\frac{1}{p}}\\
&=m\bigg(\frac{1}{(16m)^n}\sum_{\e\in \{-1,1\}^n}\sum_{y\in \Z_{8m}^n}\left|\T_{\!\n\setminus \S}f(y+2\e_{\S})-\T_{\!\n\setminus \S}f(y-2\e_{\S})\right|^p\bigg)^{\frac{1}{p}}, \label{eq:m triangle}
\end{align}
where for~\eqref{eq:m triangle} make the change of variable $y=x+2(2k+1)\e_S$  in each of the summands of~\eqref{eq:before x to y}.

For every $x\in \Z_{8m}^n$ define $h_x:\{-1,1\}^n\to \R$ by
\begin{equation}\label{eq:def hx}
\forall\, \e\in \{-1,1\}^n,\qquad h_x(\e)\eqdef f(x+2\e)-f(x-2\e).
\end{equation}
Recalling~\eqref{eq:def E operator} and~\eqref{eq:TS def}, observe that for every $(x,\e)\in \Z_{8m}^n\times \{-1,1\}^n$ and $\S\subset\n$ we have
$$
\T_{\!\n\setminus \S}f(x+2\e_{\S})-\T_{\!\n\setminus \S}f(x-2\e_{\S})=\EE_{\n\setminus \S}h_x(\e).
$$
It therefore follows from~\eqref{eq:m triangle} that
\begin{multline}\label{eq:use cube version hx}
\frac{1}{(16m)^n\binom{n}{k}}\sum_{\substack{\S\subset \n\\ |\S|=k}}\sum_{\e\in \{-1,1\}^n}\sum_{x\in \Z_{8m}^n}\frac{\left|\T_{\!\n\setminus \S}f(x+4m\e_{\S})-\T_{\!\n\setminus \S}f(x)\right|^p}{m^p}\\\le
\frac{1}{(8m)^n\binom{n}{k}}\sum_{\substack{\S\subset \n\\ |\S|=k}}\sum_{x\in \Z_{8m}^n}\left\|\EE_{\n\setminus \S}h_x\right\|_p^p\lesssim_p \sum_{j=1}^n \frac{k/n }{(8m)^n}\sum_{x\in \Z_{8m}^n}\|\partial_j h_x\|_{p}^p +\frac{(k/n)^{\frac{p}{2}} }{(8m)^n}\sum_{x\in \Z_{8m}^n}\|h_x\|_{p}^p,
\end{multline}
where in the last step of~\eqref{eq:use cube version hx} we applied Theorem~\ref{thm:cube h} with $h$ replaced by $h_x$, separately for each $x\in \Z_{8m}^n$, which we are allowed to do because   the function $h_x$ is odd, so $h_x\in L_p^0(\{-1,1\}^n)$.

Next, observe that for every $(x,\e)\in \Z_{8m}^n\times \{-1,1\}^n$ and $j\in \n$ we have
\begin{multline}\label{eq:to sum partial j hx}
|\partial_jh_x(\e)|^p\stackrel{\eqref{eq:def hx}}{=}|f(x+2\e)-f(x-2\e)-f(x+2\e-4\e_je_j)+f(x-2\e+4\e_je_j)|^p\\
\le 2^{p-1}|f(x+2\e)-f(x+2\e-4\e_je_j)|^p+2^{p-1}|f(x-2\e)-f(x-2\e+4\e_je_j)|^p.
\end{multline}
By summing~\eqref{eq:to sum partial j hx} over $(x,\e)\in \Z_{8m}^n\times \{-1,1\}^n$, we therefore see that
\begin{equation}\label{eq:4 der}
\forall\, j\in \n,\qquad \frac{1}{(8m)^n}\sum_{x\in \Z_{8m}^n}\|\partial_j h_x\|_{p}^p\le \frac{2^p }{(8m)^n}\sum_{y\in \Z_{8m}^n}|f(y+4e_j)-f(y)|^p.
\end{equation}
Since for every $y\in \Z_{8m}^n$ we have
$$
|f(y+4e_j)-f(y)|^p\le 4^{p-1}\sum_{k=1}^4|f(y+ke_j)-f(y+(k-1)e_j)|^p,
$$
it follows from~\eqref{eq:4 der} that
\begin{equation}\label{eq:done with gradient}
\frac{1}{(8m)^n}\sum_{j=1}^n\sum_{x\in \Z_{8m}^n}\|\partial_j h_x\|_{p}^p\le \frac{8^p}{(8m)^n}\sum_{j=1}^n\sum_{z\in \Z_{8m}^n}|f(z+e_j)-f(z)|^p.
\end{equation}
In the same vein to the above reasoning, for every $(x,\e)\in \Z_{8m}^n\times \{-1,1\}^n$ we have
$$
|h_x(\e)|^p\stackrel{\eqref{eq:def hx}}{\le} 4^{p-1}\sum_{k=-1}^2 |f(x+k\e)-f(x+(k-1)\e)|^p.
$$
Consequently,
\begin{equation}\label{eq:4 epsilon}
\frac{1 }{(8m)^n}\sum_{x\in \Z_{8m}^n}\|h_x\|_{p}^p\le \frac{4^p}{(16m)^n}\sum_{\e\in \{-1,1\}^n}\sum_{z\in \Z_{8m}^n} |f(z+\e)-f(z)|^p.
\end{equation}
The desired estimate~\eqref{eq:T version} now follows from a substitution of~\eqref{eq:done with gradient} and~\eqref{eq:4 epsilon} into~\eqref{eq:use cube version hx}.
\end{proof}

\begin{proof}[Proof of Theorem~\ref{thm:main}] Fixing $(x,\e)\in \Z_{8m}^n\times \{-1,1\}^n$ and $\S\subset \n$, observe that
\begin{align}\label{eq:move to T}
\nonumber|f(x+4m\e_{\S})-f(x)|^p&\le 3^{p-1}\Big(\left|\T_{\!\n\setminus \S}f(x+4m\e_{\S})-\T_{\!\n\setminus \S}f(x)\right|^p+\left|f(x)-\T_{\!\n\setminus \S}f(x)\right|^p\\&\qquad \qquad +|f(x+4m\e_{\S})-
 \T_{\!\n\setminus \S}f(x+4m\e_{\S})|^p\Big).
\end{align}
By averaging~\eqref{eq:move to T} over $(x,\e)\in \Z_{8m}^n\times \{-1,1\}^n$ and all those $\S\subset \n$ with $|\S|=k$, while using translation invariance in the variable $x$, we see that
\begin{align}
\nonumber \frac{1}{(16m)^n\binom{n}{k}}&\sum_{\substack{\S\subset \n\\ |\S|=k}}\sum_{\e\in \{-1,1\}^n}\sum_{x\in \Z_{8m}^n}\frac{\left|f(x+4m\e_{\S})-f(x)\right|^p}{m^p}\\\label{eq:T version to bound}&\lesssim_p \frac{1}{(16m)^n\binom{n}{k}}\sum_{\substack{\S\subset \n\\ |\S|=k}}\sum_{\e\in \{-1,1\}^n}\sum_{x\in \Z_{8m}^n}\frac{\left|\T_{\!\n\setminus \S}f(x+4m\e_{\S})-\T_{\!\n\setminus \S}f(x)\right|^p}{m^p}\\&\qquad+\frac{1}{m^p(8m)^n\binom{n}{k}} \sum_{\substack{\S\subset \n\\ |\S|=k}}\sum_{x\in \Z_{8m}^n} |f(x)-\T_{\!\n\setminus \S}f(x)|^p.\label{eq:perturbation to bound}
\end{align}
The quantity that appears in~\eqref{eq:T version to bound} can be bounded from above using Lemma~\ref{lem:T version with scaling}, and the quantity that appears in~\eqref{eq:perturbation to bound} can be bounded from above using Lemma~\ref{lem:T close to identity}. The resulting estimate is
\begin{multline*}\label{eq:what we want, with m}
\frac{1}{(16m)^n\binom{n}{k}}\sum_{\substack{\S\subset \n\\ |\S|=k}}\sum_{\e\in \{-1,1\}^n}\sum_{x\in \Z_{8m}^n}\frac{\left|f(x+4m\e_{\S})-f(x)\right|^p}{m^p}\\
\lesssim_p \frac{k/n}{(8m)^n}\sum_{j=1}^n\sum_{x\in \Z_{8m}^n}|f(x+e_j)-f(x)|^p+\frac{\left(\frac{k}{n}\right)^{\frac{p}{2}}+\frac{1}{m^p}}{(16m)^n}\sum_{\e\in \{-1,1\}^n}\sum_{x\in \Z_{8m}^n} |f(x+\e)-f(x)|^p.
\end{multline*}
This implies the desired estimate~\eqref{eq:desired main}, since we are assuming that $m\ge \sqrt{n/k}$.
\end{proof}

\section{Proof of Theorem~\ref{thm:cube h}}\label{sec:chaos}

Suppose that $n\in \N$ and $h:\{-1,1\}^n\to \R$. For every $k\in \{0,\ldots,n\}$ the $k$th Rademacher projection of $h$ is the function $\Rad_kh:\{-1,1\}^n\to \R$ that is given by
$$
\Rad_k h(\e)\eqdef \sum_{\substack{\A\subset \n\\ |A|=k}} \hat{h}(\A) W_\A(\e).
$$
We also have the common notation $\Rad_1h=\Rad h$. Note that $\Rad_0$ is the mean of $h$, i.e., recalling the notation~\eqref{eq:def E operator}, $\Rad_0h=\EE_{\n}h$. By a classical theorem of Bonami~\cite{Bon68}, if $\eta:\{-1,1\}^n\to \R$ is a Rademacher chaos of order at most $k$, i.e., $\hat{\eta}(\A)=0$ whenever $\A\subset \n$ is such that $|\A|>k$, then for every $p\in [2,\infty)$ we have $\|\eta\|_p\le (p-1)^{k/2}\|\eta\|_2\le p^{k/2}\|\eta\|_2$. Consequently,
$$
\|\Rad_k h\|_p\le p^{\frac{k}{2}}\|\Rad_k h\|_2\le p^{\frac{k}{2}}\|h\|_2\le p^{\frac{k}{2}}\|h\|_p,
$$
where we used the fact that (by Parseval's identity) $\|\Rad_k h\|_2\le \|h\|_2$, and that $\|h\|_2\le \|h\|_p$ since $p\ge 2$. This was a quick (and standard) derivation of the following well-known operator norm bound for $\Rad_k$, which we state here for ease of future reference.
\begin{equation}\label{eq:bonami}
\left\|\Rad_k\right\|_{p\to p}\le p^{\frac{k}{2}}.
\end{equation}

Given $\S\subset\n$ and $\alpha \in \R$, for every $h:\{-1,1\}^n\to \R$ define a function $\Delta_\S^{\!\alpha} h:\{-1,1\}^n\to \R$ by
$$
\forall\,\e\in \{-1,1\}^n,\qquad \Delta_\S^{\!\alpha} h(\e)\eqdef \sum_{\substack{\A\subset\n\\\A\cap \S\neq\emptyset}}|\A\cap \S|^\alpha \hat{h}(\A)W_\A(\e).
$$
Thus, recalling the notation~\eqref{eq:def partial j} for the hypercube partial derivatives $\partial_1,\ldots,\partial_n$, as well the notation~\eqref{eq:def Rj} for the hypercube Riesz transforms $\mathsf{R}_1,\ldots,\mathsf{R}_n$, we have the following standard identities.
$$
\forall\, j\in \n,\qquad \mathsf{R}_j=\frac12 \partial_j \Delta_{\n}^{\!-\frac12}.
$$
This means that Lust-Piquard's inequality~\eqref{eq:LP square function} can we rewritten as follows.
\begin{equation}\label{eq:move Delta}
\frac{1}{p^{3/2}} \left\|\Delta^{\! \frac12}_{\n}h\right\|_{p}\lesssim \bigg\|\Big(\sum_{j=1}^n (\partial_jh)^2\Big)^{\frac12}\bigg\|_{p}\lesssim p \left\|\Delta^{\! \frac12}_{\n}h\right\|_{p}.
\end{equation}

By Khinchine's inequality (with asymptotically sharp constant, see~\cite[Lem.~2]{PZ30}), we have
$$
\bigg\|\Big(\sum_{j=1}^n (\partial_jh)^2\Big)^{\frac12}\bigg\|_{p}\le \bigg(\frac{1}{2^n}\sum_{\d\in \{-1,1\}^n} \Big\|\sum_{j=1}^n\d_j \partial_j h\Big\|^p_{p}\bigg)^{\frac{1}{p}}\lesssim \sqrt{p} \bigg\|\Big(\sum_{j=1}^n (\partial_jh)^2\Big)^{\frac12}\bigg\|_{p}.
$$
In combination with~\eqref{eq:move Delta}, this implies that
\begin{equation}\label{eq:LP-Riesz}
\frac{1}{p^{\frac32}}\left\|\Delta^{\! \frac12}_{\n}h\right\|_{p}\lesssim  \bigg(\frac{1}{2^n}\sum_{\d\in \{-1,1\}^n} \Big\|\sum_{j=1}^n\d_j \partial_j h\Big\|^p_{p}\bigg)^{\frac{1}{p}}\lesssim p^{\frac32}\left\|\Delta^{\! \frac12}_{\n}h\right\|_{p}.
\end{equation}
For ease of future reference, we also record here the following formal consequence of~\eqref{eq:LP-Riesz}, which holds true for every $\S\subset \n$ by an application of~\eqref{eq:LP-Riesz} to the restriction of $h$ to the coordinates in $\S$.

\begin{equation}\label{eq:LP-Riesz-S}
\frac{1}{p^{\frac32}}\left\|\Delta_\S^{\! \frac12}h\right\|_{p}\lesssim  \bigg(\frac{1}{2^n}\sum_{\d\in \{-1,1\}^n} \Big\|\sum_{j\in \S}\d_j \partial_j h\Big\|^p_{p}\bigg)^{\frac{1}{p}}\lesssim p^{\frac32}\left\|\Delta_\S^{\! \frac12}h\right\|_{p}.
\end{equation}

Lemma~\ref{lem:inverse laplacian} below contains bounds on negative powers of the hypercube Laplacian $\Delta_{\n}$ that will be used later, but are more general and precise than what we actually need for the proof of Theorem~\ref{thm:cube h}: we will only use the following operator norm estimate corresponding to the case $\alpha=1/2$ of Lemma~\ref{lem:inverse laplacian}, and a worse dependence on $p$ would have sufficed for our purposes as well.
\begin{equation}\label{eq:1/2}
\forall\, p\in [2,\infty),\qquad \sup_{n\in \N}\left\|\Delta_{\n}^{\!-\frac12} \right\|_{p\to p}\lesssim \sqrt{\log p}.
\end{equation}
We include here the sharp estimates of Lemma~\ref{lem:inverse laplacian}  because they are interesting in their own right and our proof yields them without additional effort. The boundedness of negative powers of the hypercube Laplacian were studied in~\cite[Section~3]{NS02} in the context of vector valued mappings. By specializing the bounds that are stated in~\cite{NS02} to the case of real valued mappings one obtains a variant of~\eqref{eq:1/2}, but with a much worse dependence on $p$ (the resulting bound grows exponentially with $p$). The (simple) proof below of Lemma~\ref{lem:inverse laplacian} follows the strategy of~\cite{NS02} while  using additional favorable properties of real valued mappings and taking care to obtain asymptotically sharp bounds.

\begin{lemma}\label{lem:inverse laplacian} Suppose that $p\in [2,\infty)$ and $\alpha\in (0,\infty)$ satisfy
\begin{equation}\label{eq:alpha upper}
\alpha\le \frac{5+\log p}{4}.
\end{equation}
Then
\begin{equation}\label{eq:norms of alpha powers}
\sup_{n\in \N}\left\|\Delta^{\!-\alpha}_{\n} \right\|_{p\to p}\asymp\frac{(\log p)^\alpha}{2^\alpha\Gamma(1+\alpha)}.
\end{equation}
\end{lemma}

\begin{remark} Some restriction on $\alpha$ in the spirit of~\eqref{eq:alpha upper} is needed for~\eqref{eq:norms of alpha powers} to hold true, since $\lim_{\alpha\to \infty} \Delta_{\n}^{-\alpha}=\Rad$ and it is known  that $\|\Rad\|_{p\to p}\gtrsim \sqrt{p}$ for $n$ large enough (as a function of $p$).

\end{remark}

\begin{proof}[Proof of Lemma~\ref{lem:inverse laplacian}] The lower estimate
\begin{equation}\label{eq:lower for every alpha}
\sup_{n\in \N}\left\|\Delta^{\!-\alpha}_{\n} \right\|_{p\to p}\gtrsim \frac{(\log p)^\alpha}{2^\alpha\Gamma(1+\alpha)}.
\end{equation}
holds true for every $\alpha\in (0,\infty)$, without the restriction~\eqref{eq:alpha upper}.  Indeed, denote $p^*\eqdef p/(p-1)$ and observe that since $\Delta_{\n}^{\!-\alpha}$ is self-adjoint it follows by duality that~\eqref{eq:lower for every alpha} is equivalent to the estimate
\begin{equation}\label{eq:lower for every alpha q}
\sup_{n\in \N}\left\|\Delta^{\!-\alpha}_{\n} \right\|_{p^*\to p^*}\gtrsim \frac{(\log p)^\alpha}{2^\alpha\Gamma(1+\alpha)}.
\end{equation}

Fix an integer  $n\ge 2$ and consider the following $f_p^n\in L_{p^*}(\{-1,1\}^n)$, for which  $\|f_p^n\|_{p^*}=1$.
$$
\forall\, \e\in \{-1,1\}^n,\qquad f_p^n(\e)\eqdef 2^{\frac{n}{p^*}}\d_{(1,\ldots,1)}(\e)=2^{-\frac{n}{p}}\prod_{j=1}^n(1+\e_j)=2^{-\frac{n}{p}}\sum_{\A\subset \n}W_{\A}.
$$

For every $u\in (0,\infty)$ and $\alpha \in (0,\infty)$ the following identity holds true.
$$
\frac{1}{u^{\alpha}} =\frac{1}{\Gamma(\alpha)}\int_0^\infty s^{\alpha-1} e^{-su}ds.
$$
Consequently,
\begin{equation}\label{eq:laplacian identity}
\Delta^{\!-\alpha}_{\n}=\frac{1}{\Gamma(\alpha)}\int_0^\infty s^{\alpha-1} e^{-s\Delta_{\n}}(I-\Rad_0)\dd s .
\end{equation}
Note that for every $s\in (0,\infty)$ and $\e\in \{-1,1\}^n$ we have
\begin{multline}\label{eq:kappa appears}
 e^{-s\Delta_{\n}}(I-\Rad_0)f_p^n(\e)=2^{-\frac{n}{p}}\sum_{\substack{\A\subset \n\\A\neq\emptyset}}e^{-s|\A|}W_\A\\=2^{-\frac{n}{p}}\Big(\prod_{j=1}^n(1+e^{-s}\e_j)-1\Big)=-2^{-\frac{n}{p}}\left(1-(1+e^{-s})^{\kappa(\e)}(1-e^{-s})^{n-\kappa(\e)}\right),
\end{multline}
where we use the notation
$$
\forall\, \e\in \{-1,1\}^n,\qquad \kappa(\e)\eqdef \big|\{j\in \n: \e_j=1\}\big|.
$$
Since the function $k\mapsto (1+e^{-s})^{k}(1-e^{-s})^{n-k}$ is increasing on $\{0,\ldots,n\}$, it follows from~\eqref{eq:kappa appears} that
$$
\forall\, \e\in \{-1,1\}^n,\qquad \kappa(\e)\le \frac{n}{2}\implies  2^{\frac{n}{p}}e^{-s\Delta_{\n}}(I-\Rad_0)f_p^n(\e)\le -\left(1-(1-e^{-2s})^{\frac{n}{2}}\right).
$$
Recalling~\eqref{eq:laplacian identity}, it therefore follows that if $\e\in \{-1,1\}^n$ satisfies $\kappa(\e)\le n/2$ then
$$
2^{\frac{n}{p}}\left|\Delta^{\!-\alpha}_{\n}f_p^n(\e)\right|\ge \frac{1}{\Gamma(\alpha)}\int_0^\infty s^{\alpha-1}\left(1-(1-e^{-2s})^{\frac{n}{2}}\right)\dd s \ge \frac{1}{\Gamma(\alpha)}\int_0^{\frac{\log n}{2}}\frac{s^{\alpha-1}}{2}\dd s= \frac{(\log n)^\alpha}{2^{1+\alpha}\Gamma(1+\alpha)}.
$$
Hence, since $\left|\left\{\e\in \{-1,1\}^n:\ \kappa(\e)\le \frac{n}{2}\right\}\right|\ge 2^{n-1}$, we have
\begin{equation}\label{eq:lower q fpn}
\left\|\Delta^{\!-\alpha}_{\n}f_p^n\right\|_{p^*}\gtrsim \frac{2^{-\frac{n}{p}}(\log n)^\alpha}{2^{\alpha}\Gamma(1+\alpha)}.
\end{equation}
The desired estimate~\eqref{eq:lower for every alpha q} now follows by choosing $n=\lceil p\rceil$ in~\eqref{eq:lower q fpn}.

Having proven~\eqref{eq:lower for every alpha}, it remains to show that under the assumption~\eqref{eq:alpha upper} we have
\begin{equation}\label{eq:upper for  alpha not too big}
\sup_{n\in \N}\left\|\Delta^{\!-\alpha}_{\n} \right\|_{p\to p}\lesssim \frac{(\log p)^\alpha}{2^\alpha\Gamma(1+\alpha)}.
\end{equation}
To this end, observe first that the identity
$$
e^{-s\Delta_{\n}}(I-\Rad_0)= \sum_{k=1}^n \frac{1}{e^{sk}}\Rad_k
$$
implies that
\begin{equation*}
\left\|e^{-s\Delta_{\n}}(I-\Rad_0)\right\|_{p\to p} \le \sum_{k=1}^n \frac{ 1}{e^{sk}}\left\|\Rad_k\right\|_{p\to p} \stackrel{\eqref{eq:bonami}}{\le} \sum_{k=1}^n \left(\frac{\sqrt{p}}{e^s}\right)^{k}.
\end{equation*}
Hence,
\begin{equation}\label{eq:s large laplacian}
e^s> \sqrt{p}\implies \left\|e^{-s\Delta_{\n}}(I-\Rad_0)\right\|_{p\to p}\lesssim \frac{\sqrt{p}}{e^s-\sqrt{p}}.
\end{equation}
Suppose that $M\in (0,\infty)$  satisfies
\begin{equation}\label{eq:M choice}
e^M\ge \sqrt{ep}\iff M\ge \frac{1+\log p}{2}.
\end{equation}
Then, by~\eqref{eq:s large laplacian} we have
\begin{equation}\label{eq:from M}
\frac{1}{\Gamma(\alpha)} \int_M^\infty s^{\alpha-1} \left\|e^{-s\Delta_{\n}}(I-\Rad_0)\right\|_{p\to p}\dd s \lesssim \frac{\sqrt{p}}{\Gamma(\alpha)} \int_M^\infty \frac{s^{\alpha-1}}{e^s} \dd s.
\end{equation}
Due to~\eqref{eq:alpha upper} and~\eqref{eq:M choice} we have $M\ge 2(\alpha-1)$.  Since the function  $s\mapsto s^{\alpha-1}e^{-s/2}$ is decreasing on $[2(\alpha-1),\infty)\supseteq [M,\infty)$, it follows that for every $s\ge M$ we have $s^{\alpha-1}e^{-s}\le M^{\alpha-1}e^{-M/2}e^{-s/2}$. So,
\begin{equation}\label{eq:gamma tail}
\int_M^\infty \frac{s^{\alpha-1}}{e^s} \dd s\le \frac{M^{\alpha-1}}{e^{M/2}} \int_M^\infty \frac{\dd s}{e^{s/2}}\lesssim \frac{M^{\alpha-1}}{e^M}.
\end{equation}
A substitution of~\eqref{eq:gamma tail} into~\eqref{eq:from M} yields the estimate
\begin{equation}\label{eq:from M final}
\frac{1}{\Gamma(\alpha)} \int_M^\infty s^{\alpha-1} \left\|e^{-s\Delta_{\n}}(I-\Rad_0)\right\|_{p\to p}\dd s \lesssim \frac{M^{\alpha-1}\sqrt{p}}{e^M\Gamma(\alpha)}.
\end{equation}
At the same time, since for every $s\in [0,\infty)$ we have $\|e^{-s\Delta_{\n}}\|_{p\to p}\le 1$ (because $e^{-s\Delta_{\n}}$ is an averaging operator) and  $\|I-\Rad_0\|_{p\to p}\le 2$, we have
\begin{equation}\label{eq:just contraction}
\frac{1}{\Gamma(\alpha)} \int_0^M s^{\alpha-1} \left\|e^{-s\Delta_{\n}}(I-\Rad_0)\right\|_{p\to p}\dd s\lesssim \frac{1}{\Gamma(\alpha)} \int_0^M s^{\alpha-1} \dd s =\frac{M^\alpha}{\Gamma(1+\alpha)}.
\end{equation}

Making the choice
\begin{equation}\label{eq:M choice final}
M\eqdef \frac{1+\log p}{2},
\end{equation}
we see that
\begin{multline*}
\left\|\Delta^{\!-\alpha}_{\n} \right\|_{p\to p} \stackrel{\eqref{eq:laplacian identity}}{\le} \frac{1}{\Gamma(\alpha)} \int_0^\infty s^{\alpha-1} \left\|e^{-s\Delta_{\n}}(I-\Rad_0)\right\|_{p\to p}\dd s
\stackrel{\eqref{eq:from M final}\wedge \eqref{eq:just contraction}}{\lesssim} \frac{M^\alpha}{\Gamma(1+\alpha)}+ \frac{M^{\alpha-1}\sqrt{p}}{e^M\Gamma(\alpha)}\\\stackrel{\eqref{eq:M choice final}}{=}\frac{(1+\log p)^\alpha}{2^\alpha\Gamma(1+\alpha)}+ \frac{(1+\log p)^{\alpha-1}}{2^{\alpha-1}\Gamma(\alpha)\sqrt{e}}\asymp \frac{(\log p)^\alpha}{2^\alpha\Gamma(\alpha)}\left(\frac{1}{\alpha}+\frac{1}{\log p}\right)\left(1+\frac{1}{\log p}\right)^\alpha\stackrel{\eqref{eq:alpha upper}}{\asymp}
\frac{(\log p)^\alpha}{2^\alpha\Gamma(1+\alpha)}.
\end{multline*}
This is precisely the desired estimate~\eqref{eq:upper for  alpha not too big}, thus completing the proof of Lemma~\ref{lem:inverse laplacian}.
\end{proof}

\begin{lemma}\label{lem:use jensen}
Fix $n\in \N$, $p\in [1,\infty)$ and $\alpha\in \R$. Then for $h\in L_p^0(\{-1,1\}^n)$ and $S\subset \n$ we have,
$$
\left\|\mathsf{E}_{\n\setminus \S} h \right\|_{p}\le \left\|\Delta_{\S}^{\!\alpha}\Delta_{\n}^{\!-\alpha}h\right\|_p.
$$
\end{lemma}

\begin{proof} Observe that we have the following identity of operators on $L_p^0(\{-1,1\}^n)$.
\begin{equation}\label{eq:crucial laplacian identity}
\EE_{\n\setminus \S}\Delta_{\S}^{\!\alpha}\Delta_{\n}^{\!-\alpha}=\EE_{\n\setminus \S}.
\end{equation}
Indeed, if $\emptyset \neq \A\subset \n$ then $\EE_{\n\setminus \S}W_A=\1_{\{\A\subset \S\}} W_A$, and at the same time we have
$$
\EE_{\n\setminus \S}\Delta_{\S}^{\!\alpha}\Delta_{\n}^{\!-\alpha}W_{\A}=\frac{|\A\cap \S|^\alpha}{|\A|^\alpha} \EE_{\n\setminus \S}W_A=\frac{|\A\cap \S|^\alpha}{|\A|^\alpha} \1_{\{\A\subset \S\}} W_A=\1_{\{\A\subset \S\}} W_A.
$$
Consequently,
\begin{equation}\label{eq:use jensen}
\left\|\Delta_{\S}^{\!\alpha}\Delta_{\n}^{\!-\alpha}h\right\|_p^p=\EE_{\S}\EE_{\n\setminus \S} \left|\Delta_{\S}^{\!\alpha}\Delta_{\n}^{\!-\alpha}h\right|^p\ge \EE_{\S} \left|\EE_{\n\setminus \S}\Delta_{\S}^{\!\alpha}\Delta_{\n}^{\!-\alpha}h\right|^p\stackrel{\eqref{eq:crucial laplacian identity}}{=} \EE_{\S} \left|\EE_{\n\setminus \S}h\right|^p=\left\|\mathsf{E}_{\n\setminus \S} h \right\|_{p}^p,
\end{equation}
where the inequality in~\eqref{eq:use jensen} follows from Jensen's inequality ($\mathsf{E}_{\n\setminus \S}$ is an averaging operator).
\end{proof}

\begin{proof}[Proof of Theorem~\ref{thm:cube h}]
By Lemma~\ref{lem:use jensen} we have
\begin{equation}\label{eq:use norm 1 on subsets}
\bigg(\frac{1}{\binom{n}{k}} \sum_{\substack{\S\subset \n\\ |\S|=k}}\left\|\mathsf{E}_{\n\setminus \S} h \right\|_{p}^p\bigg)^{\frac{1}{p}} \le \bigg(\frac{1}{\binom{n}{k}} \sum_{\substack{\S\subset \n\\ |\S|=k}}\Big\|\Delta_{\S}^{\!\frac12}\Delta_{\n}^{\!-\frac12}h\Big\|_p^p\bigg)^{\frac{1}{p}}.
\end{equation}
By Lust-Piquard's discrete Riesz  transform inequality~\eqref{eq:LP-Riesz-S}, for every fixed $\S\subset \n$ we have
\begin{equation}\label{eq:use discrete riesz on S}
\Big\|\Delta_{\S}^{\!\frac12}\Delta_{\n}^{\!-\frac12}h\Big\|_p\lesssim p^{\frac32} \bigg(\frac{1}{2^n}\sum_{\d\in \{-1,1\}^n} \Big\|\sum_{j\in \S}\d_j \partial_j \Delta_{\n}^{\!-\frac12} h\Big\|^p_{p}\bigg)^{\frac{1}{p}}.
\end{equation}
A substitution of~\eqref{eq:use discrete riesz on S} into~\eqref{eq:use norm 1 on subsets} yields
\begin{equation}\label{eq:get rademacher upper}
\bigg(\frac{1}{\binom{n}{k}} \sum_{\substack{\S\subset \n\\ |\S|=k}}\left\|\mathsf{E}_{\n\setminus \S} h \right\|_{p}^p\bigg)^{\frac{1}{p}} \lesssim p^{\frac32}
\bigg(\frac{1}{2^n\binom{n}{k}}\sum_{\substack{\S\subset \n\\ |\S|=k}}\sum_{\d\in \{-1,1\}^n} \Big\|\sum_{j\in \S}\d_j \partial_j \Delta_{\n}^{\!-\frac12} h\Big\|^p_{p}\bigg)^{\frac{1}{p}}.
\end{equation}
For fixed $\e\in \{-1,1\}^n$,  the linear $X_p$ inequality~\eqref{eq:linearXp} with $\left\{a_j=\partial_j \Delta_{\n}^{\!-\frac12} h(\e)\right\}_{j=1}^n$ yields the estimate
\begin{multline}\label{eq:use linear Xp}
\bigg(\frac{1}{2^n\binom{n}{k}}\sum_{\substack{\S\subset \n\\ |\S|=k}}\sum_{\d\in \{-1,1\}^n} \Big|\sum_{j\in \S}\d_j \partial_j \Delta_{\n}^{\!-\frac12} h(\e)\Big|^p\bigg)^{\frac{1}{p}}\\
\lesssim \frac{p}{\log p}\bigg(\frac{k}{n}\sum_{j=1}^n \Big|\partial_j\Delta_{\n}^{\!-\frac12}  h(\e)\Big|^p+
\frac{(k/n)^{\frac{p}{2}}}{2^n}\sum_{\d\in \{-1,1\}^n} \Big|\sum_{j=1}^n\d_j \partial_j \Delta_{\n}^{\!-\frac12} h(\e)\Big|^p\bigg)^{\frac{1}{p}}.
\end{multline}
By taking $L_p$ norms with respect to $\e\in \{-1,1\}^n$, it follows from~\eqref{eq:use linear Xp} that
\begin{multline}\label{eq:use linear Xp-norms}
\bigg(\frac{1}{2^n\binom{n}{k}}\sum_{\substack{\S\subset \n\\ |\S|=k}}\sum_{\d\in \{-1,1\}^n} \Big\|\sum_{j\in \S}\d_j \partial_j \Delta_{\n}^{\!-\frac12} h\Big\|_p^p\bigg)^{\frac{1}{p}}\\
\lesssim \frac{p}{\log p}\bigg(\frac{k}{n}\sum_{j=1}^n \Big\|\partial_j\Delta_{\n}^{\!-\frac12}  h\Big\|_p^p+
\frac{(k/n)^{\frac{p}{2}}}{2^n}\sum_{\d\in \{-1,1\}^n} \Big\|\sum_{j=1}^n\d_j \partial_j \Delta_{\n}^{\!-\frac12} h\Big\|_p^p\bigg)^{\frac{1}{p}}.
\end{multline}
By Lemma~\ref{lem:inverse laplacian} we have
\begin{equation}\label{eq:laplacian bounds sum}
\bigg(\sum_{j=1}^n \Big\|\partial_j\Delta_{\n}^{\!-\frac12}  h\Big\|_{p}^p\bigg)^{\frac{1}{p}}\lesssim \sqrt{\log p} \bigg(\sum_{j=1}^n \big\|\partial_j h\big\|_{p}^p\bigg)^{\frac{1}{p}},
\end{equation}
and another application of Lust-Piquard's discrete Riesz transform inequality~\eqref{eq:LP-Riesz} shows that
\begin{equation}\label{eq:use riesz on entire cube}
\bigg(\frac{1}{2^n}\sum_{\d\in \{-1,1\}^n} \Big\|\sum_{j=1}^n\d_j \partial_j \Delta_{\n}^{\!-\frac12} h\Big\|^p_{p}\bigg)^{\frac{1}{p}}\lesssim p^{\frac32} \|h\|_p.
\end{equation}
The desired estimate~\eqref{eq:on cube p implicit} (in its slightly more refined form~\eqref{eq:on cube bext we can do in terms of p})  now follows by substituting~\eqref{eq:laplacian bounds sum} and~\eqref{eq:use riesz on entire cube} into~\eqref{eq:use linear Xp-norms}, and then substituting the resulting inequality into~\eqref{eq:get rademacher upper}.
\end{proof}

\section{Beyond $L_p$}

Fix $p\in [2,\infty)$. Following~\cite{NS14}, a Banach space $(X,\|\cdot\|_X)$ is said to be an $X_p$ Banach space  if for every $k,n\in \N$ with $k\in \n$, every $\mathsf{v}_1,\ldots,\mathsf{v}_n\in X$ satisfy
$$
\frac{1}{2^n\binom{n}{k}}\sum_{\substack{\S\subset \n\\ |\S|=k}}\sum_{\e\in \{-1,1\}^n}\Big\|\sum_{j\in \S}\e_j\mathsf{v}_j\Big\|_X^p\lesssim_X \frac{k}{n}\sum_{j=1}^n \|\mathsf{v}_j\|_X^p+\frac{(k/n)^{\frac{p}{2}}}{2^n}\sum_{\e\in \{-1,1\}^n}\Big\|\sum_{j=1}^n \e_j\mathsf{v}_j\Big\|_X^p.
$$
Inequality~\eqref{eq:linearXp} implies that $L_p$ is an $X_p$ Banach space when $p\in [2,\infty)$, and in~\cite{NS14} it was proven that for $p$ in this range  also the Schatten--von Neumann trace class $\S_p$ is an $X_p$ Banach space. Thus, due to~\cite{McC67}, there exists an $X_p$ Banach space that is not isomorphic to a subspace of $L_p$.

By~\cite{NS14} we know that any $X_p$ Banach space is also an $X_p$ metric space (see~\cite{Nao12,Bal13} for the significance of such results in the context of the Ribe program). Our proof of Theorem~\ref{thm:main} does not imply this general statement, since it relies on additional properties of the target Banach space $X$, which in our case is $L_p$. An inspection of our proof reveals that it uses only two nontrivial properties of the target space Banach $X$. Firstly, we need the following operator norm bounds.
\begin{equation}\label{eq:-1/2 laplacian}
\sup_{n\in \N} \left\|\Delta_{\n}^{\!-\frac12}\otimes I_X\right\|_{L_p(\{-1,1\}^n,X)\to L_p(\{-1,1\}^n,X)}<\infty.
\end{equation}
By~\cite[Theorem~5]{NS02}, the requirement~\eqref{eq:-1/2 laplacian} is equivalent to $X$ being a $K$-convex Banach space (for background on $K$-convexity, see the survey~\cite{Mau03}).  Secondly, we need $X$ to satisfy the following vector valued version of Lust-Piquard's inequality~\eqref{eq:LP-Riesz} for every $n\in \N$ and  $h:\{-1,1\}^n\to X$.
\begin{equation}\label{eq:riesz vector valued}
\frac{1}{2^n}\sum_{\d\in \{-1,1\}^n} \Big\|\sum_{j=1}^n\d_j \big(\partial_j\otimes I_X\big) h\Big\|^p_{L_p(\{-1,1\}^n,X)}\asymp_X \left\|\Big(\Delta^{\! \frac12}_{\n}\otimes I_X\Big)h\right\|_{L_p(\{-1,1\}^n,X)}^p.
\end{equation}

So, the argument of the present article actually shows that any $K$-convex $X_p$ Banach space $X$ that satisfies~\eqref{eq:riesz vector valued} is also an $X_p$ metric space, with the same scaling parameter as in the statement of Theorem~\ref{thm:main}. However, as we shall explain in Corollary~\ref{cor:nontrivial type} below, the validity of~\eqref{eq:riesz vector valued} already implies that $X$ is $K$-convex. This means that~\eqref{eq:-1/2 laplacian} is a consequence of~\eqref{eq:riesz vector valued} and there is no need to stipulate the validity of~\eqref{eq:-1/2 laplacian} as a separate assumption. We therefore have the following theorem.

\begin{theorem}\label{thm:generalization}
Suppose that $p\in [2,\infty)$ and that $(X,\|\cdot\|_X)$ is an $X_p$ Banach space that satisfies~\eqref{eq:riesz vector valued}.  Suppose also that $k,m,n\in \N$ satisfy $k\in \n$ and $m\ge \sqrt{n/k}$. Then every $f:\Z_{8m}^n\to X$ satisfies
\begin{multline}\label{eq:metric xp in X}
\bigg(\frac{1}{\binom{n}{k}}\sum_{\substack{\S\subset \n\\ |\S|=k}}\E\Big[\left\|f(x+4m\e_{\S})-f(x)\right\|_{X}^p\Big]\bigg)^{\frac{1}{p}}
\\\lesssim_X m\bigg(\frac{k}{n}\sum_{j=1}^n\E\Big[\|f(x+e_j)-f(x)\|_{X}^p\Big]+\left(\frac{k}{n}\right)^{\frac{p}{2}}\E\Big[\|f(x+\e)-f(x)\|_{X}^p\Big]\bigg)^{\frac{1}{p}},
\end{multline}
where the expectations are taken with respect to $(x,\e)\in \Z_{8m}^n\times \{-1,1\}^n$ chosen uniformly at random.
\end{theorem}

It seems to be quite challenging to obtain a clean and useful characterization of the class of Banach spaces that satisfy the dimension-independent vector valued discrete Riesz transform inequality~\eqref{eq:riesz vector valued}. We did verify, in collaboration with A. Eskenazis, that the Schatten--von Neumann trace class $\S_p$ satisfies~\eqref{eq:riesz vector valued} when $p\in [2,\infty)$, but in order to see this one needs to reexamine  Lust-Piquard's proof in~\cite{Lus98}  while checking in several instances that her argument could be adjusted so as to apply to $\S_p$-valued functions as well. Since including such an argument here would be quite lengthy (and mostly a repetition of Lust-Piquard's work), we postpone the justification of~\eqref{eq:riesz vector valued} when $X=\S_p$ to forthcoming work that is devoted to vector valued Riesz transforms. Due to the fact that $\S_p$ was shown to be an $X_p$ Banach space in~\cite{NS14}, Theorem~\ref{thm:generalization} holds true when $X=\S_p$, with our current proof showing that the implicit constant in~\eqref{eq:metric xp in X} (with $X=\S_p$) is $O(p^4/\sqrt{\log p})$.

It remains to prove that  if a Banach space $X$ satisfies~\eqref{eq:riesz vector valued} then $X$ is $K$-convex.   In fact, the following stronger statement holds true (see~\cite{Mau03} for background on type of Banach spaces).

\begin{proposition}\label{prop:type 2}
Suppose that $p\in [1,\infty)$ and $\alpha\in (0,1)$, and that $(X,\|\cdot\|_X)$ is a Banach space such that for every $n\in \N$ and every $h:\{-1,1\}^n\to X$ we have
\begin{equation}\label{eq:one sided vector riesz}
\frac{1}{2^n}\sum_{\d\in \{-1,1\}^n} \Big\|\sum_{j=1}^n\d_j \big(\partial_j\otimes I_X\big) h\Big\|^p_{L_p(\{-1,1\}^n,X)}\lesssim_X \left\|\big(\Delta^{\! \alpha}_{\n}\otimes I_X\big)h\right\|_{L_p(\{-1,1\}^n,X)}^p.
\end{equation}
Then $X$ has type $\frac{1}{\alpha}-\tau$ for every $\tau\in (0,1]$. In particular, if~\eqref{eq:riesz vector valued} holds true then $X$ has type $2-\tau$ for every $\tau\in (0,1]$.
\end{proposition}

By Pisier's $K$-convexity theorem~\cite{Pis82}, a Banach space $X$ has type strictly larger than $1$ if and only if $X$ is $K$-convex. We therefore have the following corollary of Proposition~\ref{prop:type 2}.
\begin{corollary}\label{cor:nontrivial type}
If $p\in [1,\infty)$ and  $(X,\|\cdot\|_X)$ is a Banach space that satisfies~\eqref{eq:one sided vector riesz} then $X$ is $K$-convex.
\end{corollary}

\begin{proof}[Proof of Proposition~\ref{prop:type 2}] Let $r_X\in [1,2]$ be the supremum over those $r\in [1,2]$ such that $X$ has type $r$. Our goal is to show that $r_X\ge 1/\alpha$. By the Maurey--Pisier theorem~\cite{MP76}, for every $n\in \N$ there exists a linear operator $\J_n:L_{r_X}(\{-1,1\}^n)\to X$ such that
\begin{equation}\label{eq:Jn assumption}
\forall\, g\in L_{r_X}(\{-1,1\}^n),\qquad \|g\|_{r_X}\le \|\J_ng\|_{X}\le 2\|g\|_{r_X}.
\end{equation}

Fixing $n\in \N$ and $g\in L_{r_X}(\{-1,1\}^n)$, for every  $\omega\in \{-1,1\}^n$ define $g_\omega\in L_{r_X}(\{-1,1\}^n)$  by
\begin{equation}\label{eq:def shift}
\forall\, \e\in \{-1,1\}^n,\qquad g_\omega(\e)\eqdef g(\omega\e)=g(\omega_1\e_1,\ldots,\omega_n\e_n).
\end{equation}
Next, define $h_g:\{-1,1\}^n\to X$ by setting
\begin{equation}\label{eq:hg}
\forall\, \omega\in \{-1,1\}^n,\qquad h_g(\omega)\eqdef \J_n g_\omega\in X.
\end{equation}
It follows from~\eqref{eq:def shift} and~\eqref{eq:hg} that
\begin{equation}\label{eq:shift in fourier}
\forall\, \omega\in \{-1,1\}^n,\qquad h_g(\omega)=\J_n\bigg(\sum_{\A\subset \n} \hat{g}(\A)W_{\A}(\omega)W_\A\bigg)=\sum_{\A\subset \n} \hat{g}(\A)W_{\A}(\omega)\J_n(W_\A).
\end{equation}

By~\eqref{eq:shift in fourier}, for every $\omega\in \{-1,1\}^n$ we have
\begin{equation}\label{eq:half laplace vector}
\big(\Delta^{\! \alpha}_{\n}\otimes I_X\big)h_g(\omega)=\sum_{\A\subset \n} \sqrt{|\A|}\cdot \hat{g}(\A)W_{\A}(\omega)\J_n(W_\A)=\J_n \big(\Delta^{\! \alpha}_{\n}g_\omega\big)=\J_n \Big( \big(\Delta^{\! \alpha}_{\n}g\big)_\omega\Big).
\end{equation}
Consequently,
\begin{equation}\label{eq:use Jn upper}
\forall\, \omega\in \{-1,1\}^n,\qquad \left\|\big(\Delta^{\! \alpha}_{\n}\otimes I_X\big)h_g(\omega)\right\|_X\stackrel{\eqref{eq:half laplace vector}\wedge\eqref{eq:Jn assumption}}{\le} 2\left\|\big(\Delta^{\! \alpha}_{\n}g\big)_\omega\right\|_{r_X}=2\left\|\Delta^{\! \alpha}_{\n}g\right\|_{r_X}.
\end{equation}

In a similar vein, it follows from~\eqref{eq:shift in fourier} that for every $\omega,\d\in \{-1,1\}^n$ we have
\begin{equation}\label{eq:shifted rademacher sum}
\sum_{j=1}^n\d_j \big(\partial_j\otimes I_X\big) h_g(\omega)=\J_n \bigg(\Big(\sum_{j=1}^n \d_j \partial_j g\Big)_\omega\bigg).
\end{equation}
Hence,
\begin{equation}\label{eq:use Jn lower}
\forall\, \omega,\d\in \{-1,1\}^n,\qquad \Big\|\sum_{j=1}^n\d_j \big(\partial_j\otimes I_X\big) h_g(\omega)\Big\|_X\stackrel{\eqref{eq:shifted rademacher sum}\wedge \eqref{eq:Jn assumption}}{\ge} \Big\|\sum_{j=1}^n \d_j \partial_j g\Big\|_{r_X}.
\end{equation}

By combining~\eqref{eq:use Jn upper} and~\eqref{eq:use Jn lower} with an application of~\eqref{eq:one sided vector riesz} to $h=h_g$, it follows that
\begin{equation}\label{eq:riesz for rx}
 \bigg(\frac{1}{2^n}\sum_{\d\in \{-1,1\}^n}\Big\|\sum_{j=1}^n\d_j \partial_j g\Big\|_{r_X}^{r_X}\bigg)^{\frac{1}{r_X}}\lesssim_p \bigg(\frac{1}{2^n}\sum_{\d\in \{-1,1\}^n}\Big\|\sum_{j=1}^n\d_j \partial_j g\Big\|_{r_X}^{p}\bigg)^{\frac{1}{p}}\lesssim_{p,X} \left\|\Delta^{\! \alpha}_{\n}g\right\|_{r_X},
\end{equation}
where the first step of~\eqref{eq:riesz for rx} uses Kahane's inequality~\cite{Kah64}. When $\alpha=1/2$, by a result of Lamberton~\cite[p.~283]{Lus98}, and for general $\alpha\in (0,1)$ by a result of the author and Schechtman~\cite[Section~5.5]{EL08}, the validity of~\eqref{eq:riesz for rx} for every $n\in \N$ and $g\in L_{r_X}(\{-1,1\}^n)$ implies that $r_X\ge 1/\alpha$.\qedhere

\subsection*{Acknowledgements} I am very grateful to Alexandros Eskenazis for many helpful discussions.

\end{proof}

\bibliographystyle{alphaabbrvprelim}
\bibliography{RieszXp}

 \end{document}